\documentclass[11pt]{article}
\usepackage{subfigure}
\usepackage{authblk}
\usepackage{hyperref}
\usepackage{float}

\usepackage{amssymb,amsthm,amsmath,mathrsfs,fullpage,multirow}
\usepackage{tikz}
\usetikzlibrary{matrix}

\renewcommand\S{\mathcal S}

\newtheorem{theorem}{Theorem}[section]

\newtheorem{proposition}[theorem]{Proposition}
\newtheorem{corollary}[theorem]{Corollary}

\theoremstyle{definition}

\newtheorem{remark}[theorem]{Remark}
\newtheorem{example}[theorem]{Example}

\makeatletter
\def\keywords{\xdef\@thefnmark{}\@footnotetext}
\makeatother

\newcommand{\thistheoremname}{}
\newtheorem*{genericthm*}{\thistheoremname}
\newenvironment{namedthm*}[1]
  {\renewcommand{\thistheoremname}{#1}%
   \begin{genericthm*}}
  {\end{genericthm*}}

\usepackage{enumerate}

\title{Pattern restricted parking functions with limited preferences}
\author[1]{Kassie Archer}
\author[1]{Robert P. Laudone}
\affil[1]{{\small Department of Mathematics, United States Naval Academy, Annapolis, MD 21402}}
\affil[ ]{{\small Email: karcher@usna.edu, laudone@usna.edu}}
\date{}

\begin{document}

\keywords{2020 \emph{Mathematics Subject Classification.} Primary 05A05; Secondary 05A15}%
\keywords{\emph{Keywords:} Parking functions; pattern avoidance}%

\maketitle

\begin{abstract}
A parking function is said to have $\ell$-limited preferences if there are at most $\ell$ copies of each element in the parking function. In this paper, we define a notion of pattern avoidance for parking functions and enumerate those $2$-limited parking functions that avoid certain patterns.
\end{abstract}

\section{Introduction}

In classical pattern avoidance, a word $\pi = \pi_1\ldots \pi_n$ contains another word (or pattern) $\sigma=\sigma_1\ldots\sigma_k$ if there is a $k$-length subsequence of $\pi$ that appears in the same relative order as $\sigma$, and $\pi$ avoids $\sigma$ if it does not contain it. This notion of pattern avoidance has been applied to traditional permutations, $k$-regular words, inversion sequences, and more. In this paper, we consider classical pattern avoidance in the set of parking functions, which are words that describe a valid list of parking preferences for $n$ cars parking in $n$ spaces.

While pattern avoidance has been considered previously for parking functions, most of this work does not treat the parking function itself as a word that contains or avoids patterns. Instead, a permutation is associated to that parking function in some way, and pattern avoidance for that permutation is considered. For instance, in \cite{J25}, the author considers pattern avoidance in the permutation describing the ultimate parking order associated to a parking function. In \cite{AP23}, the authors instead consider pattern avoidance in permutations associated to labeled Dyck paths obtained from parking functions; an association originally introduced in \cite{QR19} to study pattern avoidance in ordered set partitions.

The classical notion we consider in this paper, where a parking function is treated as a word, was also considered in \cite{JM09}. In that paper, the authors prove that certain patterns are ``strongly'' Wilf-equivalent. This notion of strong Wilf-equivalence implies some results for parking functions. For example, it implies that the number of parking functions avoiding 212 equals the number that avoid 122. However, direct enumeration of classical pattern avoidance classes for parking functions remains an open problem.

The parking functions considered in this paper are those with limited preferences, meaning that only a restricted number of cars can prefer the same spot. This fits into a broader trend of studying restricted or relaxed conditions on parking functions. For example, if one requires that whenever a car prefers spot $k \ge 2$, the spots $\{1,\dots,k-1\}$ must also be preferred by some car, these parking functions are Cayley permutations. Pattern avoidance in this setting was recently studied in \cite{CCEG}. Alternatively, allowing cars to ``back up" before moving forward recovers Naples parking functions, studied in \cite{FV}. 

Viewed as a word, our limited preference restriction means that any given letter can only occur a bounded number of times. This naturally parallels the study of pattern avoidance in multi-permutations with restricted repetitions. Examples of this include avoidance in noncrossing and nonnesting \cite{AGPS19,AL,EL24} permutations, which are $2$-regular, or canon permutations \cite{L26} which can be $k$-regular for a fixed $k \ge 1$.

In Section \ref{sec:background}, we introduce the necessary background and notation used throughout the paper. In Section \ref{sec: limited-pref}, we study parking functions with limited preferences and no other avoidance conditions, explaining a bijection with trees where each vertex has a limited number of children. Then in Section \ref{sec: 111-single}, we enumerate parking functions with 2-limited preferences that avoid patterns of length 3, while in Section \ref{sec:111-double}, we enumerate those avoiding pairs of these patterns. In Section~\ref{sec:1122}, we consider parking functions with $2$-limited preferences that avoid patterns which are permutations of $\{1,1,2,2\}$, and in Section~\ref{sec: additional} we collect some further results and directions.

\section{Background and Notation} \label{sec:background}

In classical pattern avoidance, we say that a permutation $\pi=\pi_1\ldots\pi_n$ contains a pattern $\sigma = \sigma_1\ldots \sigma_k$ if there is some $i_1<\cdots<i_k$ so that $\pi_{i_r}<\pi_{i_s}$ if and only if $\sigma_r<\sigma_s$ for any $1\leq r,s\leq k$. If $\pi$ does not contain $\sigma,$ we say $\pi$ avoids $\sigma.$ This definition has been extended to various other combinatorial objects including multi-permutations, words, Stirling permutations, Cayley permutations, etc. In this case, since repetition is allowed, the additional requirement on containment that $\pi_{i_r}=\pi_{i_s}$ if and only if $\sigma_r=\sigma_s$ is included. It is exactly this definition of containment and avoidance that we will use for parking functions in this paper. 

A word $\pi =\pi_1\pi_2\ldots\pi_n$ of size $n$ with $\pi_i\in[n]$ for all $i$ is a parking function if $|\{j : \pi_j\leq m\}|\ge m$ for all $m.$ In other words, if we sort the elements of $\pi$ from smallest to largest into a sequence $q_1q_2\ldots q_n$, then $q_i\le i$ for all $i$. For example, the words $34152$ and $33122$ are both parking functions of size 5. 

The terminology for parking functions comes from the following interpretation. 
Given a sequence of $n$ cars and a line of $n$ parking spots, we can think of $\pi_i$ as the parking spot preference for car $i$. A car will drive to its preferred parking spot and park if there is no other car, and otherwise it takes the next available spot. A parking function is a sequence of preferences where every car parks in one of the $n$ spots.

Let $\mathcal{P}_n$ denote the set of parking functions of size $n$.
We will say that a parking function $\pi=\pi_1\ldots \pi_n\in\mathcal{P}_n$ avoids a pattern $\sigma=\sigma_1\ldots \sigma_k$ (not necessarily a parking function itself, but rather a word on $\{1,2,\ldots, m\}$ for some $m\leq k$) if there is no subsequence of $\pi$ that appears in the same relative order as $\sigma$ in the classical sense, as described in the previous paragraph. For example, the parking function $454122862$ contains the pattern 121 since 454, 282, and 262 are all copies of the pattern 121, but it avoids 212 since between any two equal elements $j$, there is no element $i$ with $i<j$.
Let $p_n(S)$ denote the number of parking functions of length $n$ avoiding a set of patterns $S$.

We say a parking function $\pi=\pi_1\pi_2\ldots \pi_n$ has $\ell$-limited preferences if for each $1\leq j \leq n$, we have $|\{i : \pi_i=j\}|\leq \ell.$ In other words, each element can appear at most $\ell$ times in the parking function $\pi.$ In the language of parking function pattern avoidance, this corresponds to parking functions avoiding the pattern $1^{\ell+1}$ (i.e., the pattern composed of $\ell+1$ copies of 1). For example, the parking function 454122862 avoids $1111$ but contains $111$. We would say it has 3-limited preferences.

\section{Limited preferences} \label{sec: limited-pref}

There is a very nice formulation for the number of parking functions with $\ell$-limited preferences. In this section, we describe two ways of approaching this question, one in terms of trees and another in terms of Dyck paths. 

 Let $\mathcal{T}_n$ denote the set of rooted labeled trees on $n+1$ vertices with the root labeled 0 and the other vertices labeled with elements in $[n]$.
\begin{proposition}
    Parking functions of size $n$ with $\ell$-limited preferences are in bijection with trees in $\mathcal{T}_n$ where each vertex has at most $\ell$ children.
\end{proposition}
The proposition follows directly from well-known bijections between $\mathcal{P}_n$ and $\mathcal{T}_n$, including the Foata-Riordan bijection \cite{FR74} and the Chassaing-Marckert bijection \cite{CM01}. In the Foata-Riordan bijection, vertex $i$ is a child of vertex $j$ exactly when the $i$-th car prefers the spot immediately after where the $j$-th car actually parked. Since there are at most $\ell$ cars that prefer any particular spot, there will be at most $\ell$ children of any particular vertex. In the Chassaing-Marckert bijection, for a given parking function $\pi,$ the corresponding tree has vertex $i$ as a child of vertex $j$ if $\pi_i$ is  equal to the position of $j$ in the breadth-first search list. Since each value appears at most $\ell$ times, each vertex has at most $\ell$ children.

Since it will be useful later, we will also describe the enumeration of parking functions with $\ell$-limited preferences (i.e., those that avoid $1^{\ell+1}$) in terms of Dyck paths. Let us first note that weakly-increasing parking functions are in bijection with Dyck paths \cite{Yan2015}. Such a bijection can be described simply: the number of U's between the $(i-1)$st and $i$th D in a Dyck path is equal to the number of $i$s in the corresponding weakly-increasing parking function. Note that those Dyck paths that avoid $U^{\ell+1}$ are thus in correspondence with weakly increasing parking functions that avoid $1^{\ell+1}.$ 
We are now ready to find a generating function for $\ell$-limited preference parking functions.

Let $a_{n,k_2,\dots,k_\ell}$ denote the number of Dyck paths of length $2n$ with at most $\ell$ consecutive $U$'s and exactly $k_i$ occurrences of (maximal) $U^{i}$. For example, the Dyck path \[UUUD\,UDD\,UUUD\,UUDDDDD\,UD\,UD\] has $k_1=3, k_2 = 1$, and $k_3=2.$ Let $A(x,y_2,\dots,y_\ell) = \sum_{n,k_2,\dots,k_\ell} a_{n,k_2,\dots,k_\ell} x^n y_2^{k_2}\cdots y_\ell^{k_\ell}$ be the corresponding generating function.

\begin{theorem} \label{thm: 1-ell-avoidance}
    For $\ell \geq 2$, if we let  $A:=A(x,y_2,\dots,y_\ell)$, then
    \[
    A = 1 + xA + \sum_{i = 2}^\ell y_i(xA)^i
    \]
\end{theorem}

\begin{proof}
    Given such a Dyck path, it is either empty, of the form $UDB$ where $B$ is another such Dyck path, or $U^{i}DB_1D B_2 D \cdots D B_i$ where each $B_j$ is another such Dyck path. As we range over $i = 2,\dots,\ell$, each of these is recorded by $y_i (xA)^i$. This covers all the possible ways such a Dyck path would start and recovers the desired functional equation.
\end{proof}

\begin{theorem}\label{thm:Pell}
    For $n \geq 0$ and $\ell\geq 2$, $p_n(1^{\ell+1})$ has exponential generating function
    \[
    P_\ell(x) = \sum_{n \geq 0} \frac{p_n(1^{\ell+1})}{n!} x^n = A(x,\tfrac{1}{2!},\tfrac{1}{3!},\tfrac{1}{4!},\dots,\tfrac{1}{\ell!})
    \]
    where $A$ is defined as in Theorem \ref{thm: 1-ell-avoidance}.
\end{theorem}

\begin{proof}
    The number of parking functions with $n$ cars avoiding $1^{\ell+1}$ is exactly
    \[
    \sum_{(k_2,\dots,k_\ell)\in\mathcal J_n} \frac{n!\, a_{n,k_2,k_3,\dots,k_\ell}}{(2!)^{k_2}(3!)^{k_3}\cdots(\ell!)^{k_\ell}},
    \qquad
    \mathcal J_n=\Big\{(k_2,\dots,k_\ell)\in\mathbb Z_{\ge0}^{\ell-1}\;:\;\sum_{j=2}^\ell jk_j\le n\Big\},
    \]
    because $a_{n,k_2,k_3,\dots,k_\ell}$ is the number of weakly increasing parking functions of length $n$ in which, for each $i$, exactly $k_i$ distinct values have multiplicity $i$.  Each such weakly increasing parking function can be rearranged in $n!/\prod_i (i!)^{k_i}$ distinct ways. Dividing by $n!$ and summing over $n$ gives $A(x,\tfrac1{2!},\tfrac1{3!},\dots,\tfrac1{\ell!})$ by Theorem~\ref{thm: 1-ell-avoidance}.
\end{proof}

\begin{remark}
    Using Theorems \ref{thm:Pell} and \ref{thm: 1-ell-avoidance} in combination with Lagrange inversion, one can show that for $n \geq 1$,
    \[
    p_n(1^{\ell+1}) = (n!)^2 \sum_{\mathcal{J}_n} \frac{1}{\left(1 + \sum_{j=2}^\ell (j-1)k_j\right)! \left(n - \sum_{j=2}^\ell j k_j\right)! \prod_{j=2}^\ell \left(k_j! (j!)^{k_j}\right)}
    \]
    where $\mathcal{J}_n = \{(k_2,\dots,k_\ell) \in \mathbb{Z}^{\ell-1}_{\geq 0} \; | \; \sum_{j=2}^\ell j k_j \le n\}$.
\end{remark}

In this paper, we are mainly interested in 111-avoiding parking functions, corresponding to $\ell = 2$. In this case, we have $A(x,y) = \sum_{n,k}a_{n,k}x^ny^k$ equal to
\[
A(x,y) = \frac{1-x-\sqrt{1-2x+x^2-4x^2y}}{2x^2y}.
\]
We obtain the following result as a corollary to Theorem~\ref{thm:Pell}.
\begin{corollary} \label{thm: 111}
    For $n \geq 0$, $p_n(111)$ has exponential generating function 
    \[
    P_2(x) = \sum_{n \geq 0} \frac{p_n(111)}{n!} x^n = \frac{1-x-\sqrt{1-2x-x^2}}{x^2}
    \]
\end{corollary}

\begin{remark}\label{rem: motzkin}
    When $\ell=2$, $A(x,y)$ is a refinement of the Motzkin numbers. Indeed, the map $F\leftrightarrow UD$, $U\leftrightarrow UUD$, $D\leftrightarrow D$ is a bijection between Motzkin paths of length $n$ with $k$ up steps and Dyck paths of semi-length $n$ with no $UUU$ and exactly $k$ occurrences of $UU$. For example, $UFUDD$ corresponds to $UUDUDUUDDD$. Choosing the positions of the $2k$ non-level steps and then a Dyck path of semi-length $k$ on them gives the explicit formula
    \begin{equation}\label{eq: ank}
        a_{n,k}=\binom{n}{2k}C_k ,
    \end{equation}
    where $C_k$ is the $k$-th Catalan number. More generally, for $\ell\geq 2$ the series $A(x,1,\dots,1)$ counts Dyck paths with no $U^{\ell+1}$, equivalently ordered trees in which every vertex has at most $\ell$ children.
\end{remark}

\section{Avoiding $111$ and a single pattern of length 3} \label{sec: 111-single}

In this section, we consider those parking functions  with 2-limited preferences that avoid a single additional pattern of length three on alphabet $\{1,2\}.$ There are six such patterns. Clearly, since the reverse of a parking function with 2-limited preferences is still a parking function with 2-limited preferences, those that avoid 211 are equinumerous with those that avoid 112, and similarly with the patterns 221 and 122. That leaves four cases to consider. As we will see in Theorem~\ref{thm: 212}, those that avoid 122 are equinumerous with those that avoid 212, and similarly in Theorem~\ref{thm: 121} for those that avoid 112 and those that avoid 121. In both cases, we provide a recurrence. 

\begin{theorem} \label{thm: 212}
    Let $\sigma\in\{212,221,122\}$ and, for $1\le k\le n$, let $p_{n,k}$ denote the number of $\{111,\sigma\}$-avoiding parking functions of size $n$ whose largest entry is $k$, so that $p_n(111,\sigma)=\sum_{k=1}^n p_{n,k}$. Then $p_{1,1}=p_{2,1}=1$, $p_{2,2}=2$, and for $n\ge3$,
    \[
    p_{n,k} =
    \begin{cases}
    0 & k=1,  \\
    p_{n,k-1} + n\, p_{n-1, k-1} + (n-1)\, p_{n-2, k-1} & 1 < k < n,\\
    n\, p_{n-1}(111,\sigma) & k = n.
    \end{cases}
    \]
    In particular $p_n(111,212)=p_n(111,221)=p_n(111,122)$.
\end{theorem}

\begin{proof}
    We will focus on the case where $\sigma = 212$ since the other cases are similar. If $k = 1$ and $n \geq 3$, there are no such parking functions. If $k = n$, we can take any such parking function of size $n-1$ and add $n$ into any of the $n$ positions to obtain a parking function that still avoids $111$ and $212$, but now has $n$ as its largest element. This shows $p_{n,n} = n p_{n-1}(111,\sigma)$. We now focus on the less trivial cases when $1 < k < n$.

    We want to find the number of $\{111,212\}$-avoiding parking functions of size $n$ whose largest element is $k$. First, we can take any $\{111, 212\}$-avoiding parking function $\pi$ of length $n$ with $k-1$ as its largest element and we can replace all the $(k-1)$'s with $k$'s. This is still a parking function. Since $\pi$ avoids $111$, at most two entries of $\pi$ equal $k-1$, so at least $n-2\ge k-1$ entries are smaller than $k-1$. As a result, $|\{i:\pi_i\le m\}|\ge m$ still holds for $m=k-1$, and the counts for every other $m$ are unchanged. Since the relabelling is order preserving, the new word still avoids $111$ and $212$. This accounts for all the $\{111,212\}$-avoiding parking functions of length $n$ with $k$ as the largest element where no $k-1$ is present. There are $p_{n,k-1}$ such parking functions.

    Now we need to account for the $\{111,212\}$-avoiding parking functions that do contain a $k-1$. We first consider those with a single copy of $k$. These parking functions come from inserting a $k$ into any position of a $\{111,212\}$-avoiding parking function of size $n-1$ with largest element $k-1$. There are therefore $n p_{n-1,k-1}$ such parking functions. Next, we consider the parking functions with two copies of $k$; this is the only other case because of the $111$-avoidance condition. To avoid $212$, since $k$ is the maximal element, we must have the two copies of $k$ occur consecutively. We can therefore construct these parking functions by adding $kk$ into any position of a $\{111,212\}$-avoiding parking function of size $n-2$ whose largest element is $k-1$. There are $(n-1)p_{n-2,k-1}$ such parking functions.

    This accounts for all the possibilities and establishes the desired recursion when $1 < k < n$. The proof is almost identical when $212$ is replaced by either $122$ or $221$. The only difference is in the final step, instead of the two $k$ terms occurring consecutively, one of the $k$'s must occur at the beginning or the end of the parking function. The recursion remains the same.
\end{proof}

\begin{theorem} \label{thm: 121}
    If we let $p_n$ denote the number of $\{111, \sigma\}$-avoiding parking functions, where $\sigma \in \{121,211,112\}$, then $p_n = \sum_{k=1}^n p_{n,k}$ where $p_{1,1} = 1$, $p_{0} = 1$ and $p_{m,0} = 0$ for all $m$ and for $n \ge 2$:
    \[
    p_{n,k} =
    \begin{cases}
    (n-1) p_{n-2} & k = 1\\
    np_{n-1,k-1} +
    p_{n,k-1}  - np_{n-1,k-2} - (n-1) p_{n-2,k-2} & 1 < k <n \\
    n! & k=n.
    \end{cases}
    \]
\end{theorem}

\begin{proof}
    We will consider the case where $\sigma = 121$, the others are similar. Here, $P_{n,k}$ denotes the $\{111,121\}$-avoiding parking functions of size $n$ whose set of used letters is $S$, where $\{1,\dots,k\} \subset S$ and $k+1 \not\in S$. We let $p_{n,k} = |P_{n,k}|$. When $k = n$, we get permutations, so there are $n!$ options. When $k = 1$, we must have two copies of $1$ because the parking function cannot contain $2$. These are obtained by taking any $\{111,121\}$-avoiding parking function of length $n-2$, shifting all its entries up by two and inserting $11$ anywhere. The $11$ must be consecutive to avoid $121$. There are $(n-1)p_{n-2}$ such parking functions. We now consider the case where $1 < k < n$.

    First, we consider the parking functions with a single copy of $1$. We can obtain these by taking any parking function in $P_{n-1,k-1}$, shifting all of the entries up by $1$ and inserting a $1$ into any position. There are $n p_{n-1,k-1}$ such functions. These account for all the desired parking functions with a single copy of $1$.

    We now claim that there are $p_{n,k-1} - np_{n-1,k-2} - (n-1) p_{n-2,k-2}$ parking functions in $P_{n,k}$ with two copies of $1$. To see this, let $X$ denote the set of $\pi \in P_{n,k}$ with two copies of $1$ and let $Y$ denote the set of $\pi \in P_{n,k-1}$ with two copies of $1$. By the previous paragraph we have $|Y| = p_{n,k-1} - n p_{n-1,k-2}$. Now, we split $Y$ into two disjoint subsets, $Y = Y' \sqcup Y''$ where $Y'$ consists of those parking functions that contain $k+1$ and $Y''$ consists of those that do not. We claim $|X| = |Y'|$ and $|Y''| = (n-1) p_{n-2,k-2}$ which gives
    \[
    |X| = p_{n,k-1} - n p_{n-1,k-2} - (n-1) p_{n-2,k-2}
    \]
    as desired. 

    For the first claim, send $\pi \in X$ to the word $\pi'$ obtained by replacing every $k$ by $k+1$, we can do this because $k+1 \not\in S$. Since $k+1 \notin S$, this relabelling is order preserving, so $\pi'$ still avoids $111$ and $121$. It is also still a parking function. Indeed, the number of entries at most $m$ is unchanged for every $m \neq k$, while $|\{j : \pi'_j \le k\}| = |\{j : \pi_j \le k-1\}| \ge k$, because $\pi$ contains $1$ twice and each of $2,\dots,k-1$ at least once. But by construction $\pi'$ contains $\{1,\dots,k-1\}$ and $k+1$ but not $k$, so $\pi' \in Y'$. The inverse map replaces every $k+1$ by $k$, this shows $|X| = |Y'|$.

    Now to prove $|Y''| = (n-1) p_{n-2,k-2}$, let $\pi \in Y''$ so that $\pi$ contains $1$ twice and neither $k$ nor $k+1$. To avoid $121$, the two $1$s are adjacent, so we can delete them and lower every entry in $\{2,\dots, k-1\}$ by one and every entry greater than or equal to $k+2$ by two (possible since $\pi$ does not contain $k$ or $k+1$). The resulting word, $\tau$, is a parking function. Indeed, for $m \le k-2$, $|\{j:\tau_j\le m\}| = |\{j:\pi_j\le m+1\}| - 2 \ge m$, since $\pi$ contains $1$ twice and each of $2,\dots,m+1$ at least once and for $m \ge k-1$, $|\{j:\tau_j\le m\}| = |\{j:\pi_j\le m+2\}| - 2 \ge m$. Since $\tau$ contains $\{1,\dots,k-2\}$ but does not contain $k-1$, we conclude $\tau \in P_{n-2,k-2}$. There were $n-1$ places the repeated $1$s could appear in the original parking function. 
    
    Conversely, given $\pi' \in P_{n-2,k-2}$ we can add $11$ in any of the $n-1$ positions of $\pi'$, shift all the entries $\{1,\dots,k-2\}$ up by one and all the entries greater than or equal to $k$ up by two and we reverse this process. This shows $|Y''| = (n-1) p_{n-2,k-2}$ and establishes the desired recursion. When $k = 2$ both $Y''$ and $P_{n-2,0}$ are empty, and the formula reduces to $p_{n,2} = n\, p_{n-1,1} + p_{n,1}$.

    The cases where $\sigma = 211$ or $112$ are essentially identical except in the final step. Wherever we used that two copies of $1$ must be adjacent, we use instead that the first (resp. last) copy of $1$ must be the first (resp. last) entry of the word, while the other copy is unconstrained. In each case there are again $n-1$ ways to place the pair.
\end{proof}

\section{Avoiding $111$ and a pair of patterns of length 3}\label{sec:111-double}

For this section, we recall that $a_{n,k}$ denotes the number of Dyck paths of semi-length $n$ with at most $2$ consecutive $U$ steps and exactly $k$ occurrences of $UU$.

\begin{theorem} \label{thm: 112-121}
    For $n \geq 1$, $p_n(111,112,211) = p_n(111,112,121) = p_n(111,211,121) = p_n$ where $p_n=\sum_k p_{n,k}$, with $p_{1,1} = 1$ and for $n \ge 2$,
    \[
    p_{n,k} = 
    \begin{cases}
         p_{n-2} & k = 1,\\
        np_{n-1,k-1} + \sum_{j=k-1}^{n-2} p_{n-2,j} & 1 < k < n,\\
        n! & k = n.
    \end{cases}
    \]
    In particular, all three sets are Wilf-equivalent.
\end{theorem}

\begin{proof}
    Let $P_{n,k}$ denote the set of $\{111, 112,211\}$-avoiding parking functions of size $n$ whose set of used letters is $S$, where $\{1,\dots,k\} \subset S$ and $k+1 \not\in S$. We let $p_{n,k} = |P_{n,k}|$. When $k = 1$, we cannot have a single copy of $1$ and skip $2$, so the only case to consider is when we have two copies of $1$. In this case, the $1$'s must appear at the beginning and end of the parking function. We can then take any parking function of length $n-2$ that avoids $111, 112,$ and $211$, shift all the entries up by two and place it between the two $1$'s. This shows $p_{n,1} = p_{n-2}$.

    When $k = n$, every element of $P_{n,n}$ is a permutation, so $p_{n,n} = n!$. The argument below also covers this case because the second sum will be empty, but we separate to make the argument more straightforward. Now if $1 < k < n$, we have two cases to consider. Either, the parking function contains a single copy of $1$ or two copies of $1$. In the first case, there are $np_{n-1,k-1}$ options because we can take any parking function in $P_{n-1,k-1}$, shift all the entries up by one and add a $1$ anywhere.

    In the second case, if $\pi \in P_{n,k}$ has two copies of $1$ they must appear at the beginning and end of $\pi$ to avoid $112$ and $211$. For the entries between the two $1$'s we can take any parking function in $P_{n-2,j}$ for $k-1 \leq j \leq n-2$, and shift its entries less than $k$ up by one and its entries greater than or equal to $k$ up by two. The $j$ keeps track of where the next jump occurs, making this an easily reversible process.

    For the other cases $\{111,112,121\}$ and $\{111,211,121\}$, the only change is that the two copies of $1$ must be adjacent and at the end (resp.\ the beginning) of $\pi$; again there is exactly one way to place them.
\end{proof}

\begin{theorem} \label{thm: 122-221}
    For $n\geq 1$, $p_n(111,122,221)=p_n(111,212,221) = p_n(111,212,122) =p_n$ where $p_n=\sum_{k}p_{n,k}$ with $p_0=p_{0,0}=1$, and
    \[
    p_{n,k} = \begin{cases}
         p_{n,k-1} + n p_{n-1,k-1} + p_{n-2,k-1} & \text{if} \quad \lceil \tfrac{n}{2}  \rceil \le k < n\\
          n p_{n-1} & \text{if}\quad k=n \\
          0 & \text{otherwise.}
    \end{cases}\]
\end{theorem}

\begin{proof}
    Let $p_{n,k}$ be the number of $\{111,122,221\}$-avoiding parking functions with maximal entry $k$.
    If we wish to form $p_{n,k}$, we can add $k$ to any $\{111,122,221\}$-avoiding parking function of length $n-1$ and smaller maximal entry than $k$. We can add this $k$ into any position without creating a $221$ or $122$ pattern. Or we could add two $k$'s to any $\{111,122,221\}$-avoiding parking function of length $n-2$ and smaller maximal entry than $k$. We can, however, only add these two copies of $k$ as the first and last entry, otherwise we would create a $122$ or $221$. This shows that for $k<n$, 
    \[
    p_{n,k} = n\sum_{m<k} p_{n-1,m} + \sum_{m<k} p_{n-2,m},
    \]
    and subtracting the corresponding expression for $p_{n,k-1}$ gives the stated recursion. When $k=n$ the second sum is absent, so $p_{n,n}=n\,p_{n-1}$. The lower bound on $k$ comes from the $111$ avoidance.
\end{proof}

\begin{theorem} \label{thm: 121-212}
    For $n \geq 1$, $p_n(111,121,212) = \sum_{k=0}^n (n-k)!\,a_{n,k}$.
\end{theorem}

\begin{proof}
    Avoiding $121$ and $212$ forces all repeated elements to appear consecutively. If we take a weakly increasing parking function that avoids $111$ with $k$ repeated elements, we can permute all the $n-k$ distinct entries in $(n-k)!$ ways, giving $(n-k)!a_{n,k}$ possible parking functions. Summing over the possible number of repeated elements yields the stated result.
\end{proof}

\begin{theorem} \label{thm: 211-122}
    For $n \geq 1$, $p_n(111,211,122) = p_n(111,112,221) = \frac{(n+1)!}{2}$.
\end{theorem}

\begin{proof}
    First, notice such a parking function can have at most one repeated entry. If two entries $a<b$ repeat, this means the parking function contains one of $aabb, abab, baba, baab, abba, bbaa$, all of which contain either a $211$ or $122$ pattern.

    The number of parking functions with no repeated entries is $n!$. If we are allowed to have an entry occur at most two times, the number of weakly increasing parking functions with one entry repeating two times is $\binom{n}{2}$. Indeed, select two entries from $n$, let the repeated entry be the minimal entry we select, remove the other entry we select and keep the remaining $n-2$ entries. 
    
    After selecting the entries, we must arrange them. The initial entry of the parking function must be the repeated entry. We can then permute the remaining entries in any way. This means there are $\binom{n}{2} (n-1)!$ such parking functions with a repeat. In total, we find $p_n(111,211,122) = n! + \binom{n}{2} (n-1)! = \frac{(n+1)!}{2}$.
    \end{proof}

\begin{theorem} \label{thm: 112-122}
    For $n\geq 1$, $p_n=p_n(111,112,122)=p_n(111,112,212)=p_n(111,121,122)$ where $p_n=\sum_{k,\ell} p_{n,k,\ell}$ with 
    \[
    p_{n,k,\ell}=
    \begin{cases}
        (\ell-1)\sum_{m<k}p_{n-1,m,\ell-1} + \sum_{j\geq \ell-1, m<k} p_{n-2,m,j} & 2 \le k < n\\
        (\ell-1)\sum_{m<n}p_{n-1,m,\ell-1} & k = n.
    \end{cases}
    \]
    When $k = 1$, $p_{1,1,2} = 1$, $p_{2,1,2} = 1$, and for all other values of $n$ and $\ell$, $p_{n,1,\ell} = 0$.
\end{theorem}

\begin{proof}
    We focus on the case of avoiding $112$ and $212$, the others are similar and we will remark on the similarities at the end of the proof. Let $p_{n,k,\ell}$ denote the number of $\{111,112,212\}$-avoiding parking functions of size $n$, largest element $k$, and $\ell$ places where one could insert a new largest element. We will sometimes refer to these as admissible positions. By definition we must have $2 \le \ell \le n+1$, since it is always possible to insert a new largest element at the start of the parking function or in the second position. Furthermore, in this case, the admissible positions must all appear consecutively at the beginning of the parking function. Indeed, reading from left to right, we can place a new largest element up until we have our first repeat occurrence of any element. As soon as we have a repeated entry, we cannot place a larger element after that spot. This means $\ell$ actually tracks the position of the first repeated element.
    
    If $k = 1$, $p_{n,1,\ell} = 0$ when $n > 2$ to avoid $111$, so we only have a few cases to consider: $p_{1,1,2} = 1$, $p_{2,1,2} = 1$ and $p_{2,1,3} = 0$. Now assume $k > 1$.

    First, we consider how many of these parking function have a single occurrence of $k$. One could build such a parking function by inserting a single $k$ into a parking function with maximal entry less than $k$. The parking functions we start with will therefore have size $n-1$ and largest element $m$ for some $m < k$. If we want the resulting parking function to have $\ell$ admissible positions, we need the parking function we start with to have $\ell-1$ admissible positions. Indeed, if we place $k$ into any of those $\ell-1$ places, we could then place a new largest element on either side of $k$ and in any of the remaining $\ell-2$ places we did not insert $k$. This means the parking functions we desire with a single occurrence of $k$ are enumerated by $(\ell-1) \sum_{m < k} p_{n-1,m,\ell-1}$.

    Next, for $k < n$ we consider how many of these parking function contain two copies of $k$. We need $k < n$ because $n$ cannot appear twice in any parking function; in the case $k = n$, we only get the first sum. To avoid $212$, the two occurrences of $k$ must be consecutive. We will similarly build such a parking function by inserting two consecutive copies of $k$ into a parking function with maximal entry less than $k$ and size $n-2$. Since $\ell$ records the position of the first repeated element and we are introducing a repetition, this means we must add our two copies of $k$ into positions $\ell-1$ and $\ell$. In order to be able to add a largest element in these positions, the first repetition in our starting parking function must occur in a position $\ge \ell-1$. In all, this means the number of desired parking functions with two occurrences of $k$ is precisely $\sum_{j \ge \ell-1, m < k} p_{n-2,m,j}$ as desired.

    Now if we instead considered $112,122$ avoidance, the only change in the argument is when we add two copies of the largest element. In this case instead of appearing consecutively, one must appear at the start of the parking function and we place the second entry in position $\ell$. 
    
    For $121,122$ avoidance the admissible positions look different, so we need to explain this case in more detail. A new largest entry creates a $121$ exactly when some value occurs both before and after it, and it cannot create a $122$. This means position $p$ is admissible if and only if $\pi_1\cdots\pi_{p-1}$ and $\pi_p\cdots\pi_n$ have no value in common. In particular positions $1$ and $n+1$ are always admissible, but the admissible positions don't have to be consecutive anymore. For example, for $\pi=3311$ they are $1$, $3$ and $5$. Also, now two copies of the largest entry don't have to be adjacent, e.g. $3123$ avoids $111$, $121$
    and $122$. Despite this, interestingly, the recursion is the same. If $k$ occurs once, it was inserted into an admissible position of a word $\tau$ with $\ell-1$ admissible positions, and since $k$ occurs only once this creates exactly one new admissible position, the one immediately after $k$, while preserving every other position. This gives the first term. If $k$ occurs twice, avoiding $122$ forces the first copy to be $\pi_1$, and avoiding $121$ forces the second copy into an admissible position of the remaining word $\tau$. Conversely, every such placement avoids $111$, $121$ and $122$. Now, suppose the admissible positions of $\tau$ are $a_1<\dots<a_j$ and the second $k$ is inserted at $a_i$. In $\pi$, position $1$ is admissible, no position between the two copies of $k$ is, the position immediately after the second $k$ is admissible because $a_i$ was, and each later position is admissible if and only if the corresponding position of $\tau$ was. So the number of admissible positions of $\pi$ is $1+1+(j-i)=j-i+2$, and for each $\ell\ge2$ every $\tau$ with $j\ge\ell-1$ gives exactly one $\pi$ with $\ell$ admissible positions, namely $i=j-\ell+2$. This is the second term.
\end{proof}

\section{Avoiding 111 and permutations of $\{1,1,2,2\}$} \label{sec:1122}

In this section, we consider those parking functions with $2$-limited preferences that avoid a subset of the permutations of $\{1,1,2,2\}$. Our results are summarized in Table~\ref{tab:111}.

\renewcommand{\arraystretch}{1.6}
\begin{table}[H]
    \centering
    \begin{tabular}{|c|c|}
    \hline
        Set of patterns $\Lambda$ & $p_n(111,\Lambda)$   \\ \hline \hline
        $\{1221,2112\}$ & \multirow{2}{*}{$\displaystyle\sum_{k=0}^{\lfloor n/2 \rfloor} \frac{n!}{(k+1)!} a_{n,k}$} \\ \cline{1-1}
        $\{1212, 2121\}$ &  \\ \hline
        $\{1122,2211\}$ & $\displaystyle\sum_{k=0}^{\lfloor n/2 \rfloor} (k!)^2 (n-2k)! \binom{n}{2k} a_{n,k}$ \\ \hline
        $\{1221, 2112,1212,2121\}$ & $\displaystyle\sum_{k=0}^{\lfloor n/2 \rfloor} k! (n-2k)! \binom{n}{2k} a_{n,k}$ \\ \hline
    \end{tabular}
    \caption{Here $a_{n,k}$ denotes the number of Dyck paths of length $2n$ with at most $2$ consecutive $U$'s (i.e. no occurrence of $UUU$) and exactly $k$ occurrences of $UU$, see \eqref{eq: ank} in Remark \ref{rem: motzkin}.}
    \label{tab:111}
\end{table}

All of these results follow from a general result linking these parking functions to the number of multipermutations of $\{1,1,2,2,\dots,n,n\}$ that avoid a subset of permutations of $\{1,1,2,2\}$.

\begin{theorem} \label{thm: repeatedAvoidance}
    If $\Lambda$ is a subset of the permutations of $\{1,1,2,2\}$, and $M_k(\Lambda)$ is the number of permutations of $[k] \cup [k]$ avoiding $\Lambda$, then 
    \[
    p_n(111,\Lambda) = \sum_{k=0}^{\lfloor n/2 \rfloor} M_k(\Lambda) (n-2k)! \binom{n}{2k} a_{n,k}
    \]
    where $a_{n,k}$ is the number of Dyck paths with no occurrence of $UUU$, and exactly $k$ occurrences of $UU$.
\end{theorem}

\begin{proof}
    There are $a_{n,k}$ weakly increasing parking functions with $k$ repeated entries avoiding $111$. The number of ways to permute the repeated entries while still avoiding $\Lambda$ is $M_k(\Lambda)$. We can then permute the non-repeated entries in $(n-2k)!$ ways. We can insert the non-repeated entries anywhere among the repeated entries without creating a pattern in $\Lambda$ because $\Lambda$ consists of permutations of $[2] \cup [2]$. There are $\binom{n}{n-2k} = \binom{n}{2k}$ ways to insert the non-repeated entries. Summing over the number of possible repeated entries gives the result.
\end{proof}

We now consider a non-nesting or non-crossing condition on the cars. One can interpret the non-nesting condition as: no more than two cars can prefer the same spot and if two cars prefer the same spot, no two cars between them prefer the same spot.

\begin{theorem}
    Let $n\geq 1.$
    The number of parking functions avoiding $111$ and $\Lambda = \{1221, 2112\}$ or $\Lambda = \{1212, 2121\}$ is
\[
p_n(111,\Lambda) = \sum_{k=0}^{\lfloor n/2 \rfloor} \frac{n!}{(k+1)!} a_{n,k}.
\]
\end{theorem}

\begin{proof}
    We can use Theorem \ref{thm: repeatedAvoidance} here. In this case $\Lambda = \{1221, 2112\}$ or $\{1212,2121\}$ and so, in both cases, $M_k(\Lambda) = k!C_k$, since $C_k$ counts both the number of non-nesting  and the number of non-crossing matchings on $k$ elements. Simplifying
    \[
    \sum_{k=0}^{\lfloor n/2 \rfloor} k! C_k (n-2k)! \binom{n}{2k} a_{n,k}
    \]
    gives the desired result.
\end{proof}

\begin{theorem}
    The number of parking functions avoiding $111$ and $\Lambda = \{1122, 2211\}$ is
    \[
    p_n(111,\Lambda) = \sum_{k=0}^{\lfloor n/2 \rfloor} (k!)^2 (n-2k)! \binom{n}{2k} a_{n,k}.
    \]
\end{theorem}

\begin{proof}
    Using Theorem \ref{thm: repeatedAvoidance} we have to show that $M_k(1122,2211) = (k!)^2$. If a permutation of $[k] \cup [k]$ avoids $1122$ and $2211$ this means the first occurrence of every entry must appear before we can have any second occurrence. There are $k!$ ways to arrange the first entries and $k!$ ways to arrange the second entries.
\end{proof}

\begin{theorem}
    The number of parking functions avoiding $111$ and $\Lambda = \{1221, 2112,1212,2121\}$ is
    \[
    p_n(111,\Lambda) = \sum_{k=0}^{\lfloor n/2 \rfloor} k! (n-2k)! \binom{n}{2k} a_{n,k}.
    \]
\end{theorem}

\begin{proof}
    Using Theorem \ref{thm: repeatedAvoidance} we have to show that $M_k(\Lambda) = k!$. If a permutation of $[k] \cup [k]$ avoids $\Lambda$ this means the first and second occurrence of entry must appear consecutively. There are $k!$ ways to arrange the consecutive entries.
\end{proof}

\section{Additional Results} \label{sec: additional}

The methods in this paper can be extended in many ways. In this section we outline a few of them, proving some results along the way.

\subsection{Limited preferences and selected triples}

The first direction we consider is limited preference paired with avoiding select triples. For example,

\begin{theorem}
    Let $\Lambda = \{111,112,211,1221\}$. For $n \geq 1$, $p_n(\Lambda) = s(n+1,2)$, where $s(m,2)$ denotes the unsigned Stirling number of the first kind.
\end{theorem}

\begin{proof}
    We will show that for $n \geq 1$, $p_1 = 1$ and $p_{n+1} = (n+1)p_{n} + n!$, which is precisely the recursion $s(n+2,2) = (n+1)s(n+1,2) + n!$ satisfied by $s(n+1,2)$.

    It is clear that $p_1 = 1$. Now, notice any such parking function of length $n+1$ can either have one or two copies of $1$. If there is only one copy of $1$, we can form these parking functions by adding a $1$ to $\pi \in P_n(\Lambda)$. Indeed, we could insert $1$ into any of the $n+1$ positions and shift all the other entries up by one. 
    
    If there are two copies of $1$, they must be the first and last entry to avoid $112$ and $211$. The other entries must then have no repeats, otherwise the parking function contains $1221$. To fill these entries, we can first select a permutation of size $n-1$ in $(n-1)!$ ways. We can then choose at most one of entries, and all entries after, to increase by one. There are $n$ choices here, leaving a total of $n!$ options for the entries between the $1$'s.
\end{proof}

\begin{example}
    The parking function $\pi = 125374891 \in P_{9}(111,112,211,1221)$. Since $\pi_1 = \pi_{9} = 1$, the entries between the $1$'s must not have any repeats. These entries in the middle correspond to the permutation $1425367 \in \S_7$ where we shift all the entries $\geq 5$ up by one.
\end{example}

\subsection{Fewer limitations on preferences}

One other direction to consider is allowing less limitations, or no limitations, on the preferences. In a forthcoming paper, we prove more results in this vein. As an example, the following is a generalization of Theorem \ref{thm: 211-122},

\begin{theorem} \label{thm: 221-112-1k} For $n\geq 1$ and $\ell\ge1$,
    \[
    p_n(1^{\ell+1},221,112) \;=\; n!\;+\sum_{j=2}^{\ell}\binom{n}{j}\,(n-j+1)! = n!\Big(1+\sum_{j=2}^{\ell}\frac{n-j+1}{j!}\Big).
    \]
\end{theorem}

\begin{proof}
    First, notice such a parking function can have at most one repeated entry. If two entries $a,b$ repeat, this means the parking function contains one of $aabb, abab, baba, baab, abba, bbaa$, all of which contain either a $221$ or $112$ pattern.

    The number of parking functions with no repeated entries is $n!$, this corresponds to $j=1$. If we are allowed to have an entry occur at most $j$ times, the number of weakly increasing parking functions with one entry repeating $j$ times is $\binom{n}{j}$. Indeed, if we select $j$ entries from $n$, let the repeated entry be the minimal entry we select, remove the other entries we select and keep the remaining $n-j$ entries. 
    
    After selecting the entries, we must arrange them. Let $a$ be the repeated entry. Any entry following the second copy of $a$ forms a $112$ (if larger) or a $221$ (if smaller) with the first two copies of $a$, so every entry after the second copy of $a$ equals $a$. This means the parking function is $w\,a^{\,j-1}$, where $w$ is any of the $(n-j+1)!$ orderings of the $n-j$ other entries together with one copy of $a$. It is clear that every such word avoids $221$ and $112$. Summing over $j$ gives the first formula, and the second follows from $\binom nj(n-j+1)!=n!\,(n-j+1)/j!$.
\end{proof}

\begin{corollary}
For $n\geq 1,$
    $p_n(111,221,112) = n!+ (n-1)! \binom{n}{2} = \frac{(n+1)!}{2}$.
\end{corollary}

\begin{proof}
    This follows by letting $\ell = 2$ in Theorem \ref{thm: 221-112-1k}, recovering Theorem \ref{thm: 211-122} up to symmetry.
\end{proof}

\bibliographystyle{amsplain}

\begin{thebibliography}{99}

\bibitem{AP23} A. Adeniran and L. Pudwell, Pattern avoidance in parking functions, \textit{Enumer. Comb. Appl.} \textbf{3} (2023), no.~3, Art.~S2R17.

\bibitem{AGPS19} K. Archer, A. Gregory, B. Pennington, and S. Slayden, Pattern restricted quasi-Stirling permutations, \textit{Australas. J. Combin.} \textbf{74} (2019), no.~3, 389--407.

\bibitem{AL} K. Archer and R. Laudone, Pattern avoidance in non-crossing and non-nesting permutations, \textit{Enumer. Comb. Appl.} \textbf{6} (2026), no.~1, Art.~S2R1.

\bibitem{CCEG} G. Cerbai, A. Claesson, D. C. Ernst, and H. Golab, Pattern-avoiding Cayley permutations via combinatorial species, \textit{Australas. J. Combin.} \textbf{94} (2026), 267--299.

\bibitem{CM01} P. Chassaing and J.-F. Marckert, Parking functions, empirical processes, and the width of rooted labeled trees, \textit{Electron. J. Combin.} \textbf{8} (2001), no.~1, Art.~R14.

\bibitem{EL24} S. Elizalde and A. Luo, Pattern avoidance in nonnesting permutations, \textit{Discrete Math. Theor. Comput. Sci.} \textbf{27} (2025), no.~1, \#1.

\bibitem{FV} L. Ferrari and F. Verciani, A new approach to Naples parking functions through complete parking preferences, and its enumerative consequences, \textit{Electron. J. Combin.} \textbf{33} (2026), no.~3, \#P3.54.

\bibitem{FR74} D. Foata and J. Riordan, Mappings of acyclic and parking functions, \textit{Aequationes Math.} \textbf{10} (1974), 10--22.

\bibitem{JM09} V. Jel\'inek and T. Mansour, Wilf-equivalence on $k$-ary words, compositions, and parking functions, \textit{Electron. J. Combin.} \textbf{16} (2009), no.~1, Art.~R58.

\bibitem{L26} R. Laudone, Pattern avoidance in canon permutations, arXiv:2608.21351 (2026).

\bibitem{QR19} D. Qiu and J. Remmel, Patterns in words of ordered set partitions, \textit{J. Comb.} \textbf{10} (2019), no.~3, 433--490.

\bibitem{Yan2015} C. H. Yan, Parking functions, in \textit{Handbook of Enumerative Combinatorics} (M. B\'ona, ed.), CRC Press, 2015, pp. 835--893.

\bibitem{J25} J. Yan, Results on pattern avoidance in parking functions, \textit{Enumer. Comb. Appl.} \textbf{5} (2025), no.~1, Art.~S2R2.

\end{thebibliography}

\noindent {\tiny \emph{The views expressed in this paper are those of the authors and do not reflect the official policy or position of the U.S. Naval Academy, Department of the Navy, the Department of Defense, or the U.S. Government.} \par}

\end{document}